\documentclass[12pt,reqno]{amsart}
\usepackage{amssymb,amsmath,amsthm,amsfonts,fancyhdr}
\usepackage{tikz}
\usepackage{indentfirst}
\usepackage{mathrsfs}
\usepackage{setspace}
\usepackage{color}
\usepackage{cite}
\usepackage{bm}
\usepackage{geometry}
\numberwithin{equation}{section}

\newtheorem{thm}{Theorem}[section]
\newtheorem{defn}[thm]{Defnition}
\newtheorem{lemma}[thm]{Lemma}

\newtheorem{re}[thm]{Remark}

\renewcommand{\det}{\mbox{det}}

  \numberwithin{equation}{section}
  \numberwithin{figure}{section}

\begin{document}
\title[quasiconformal mappings and a Bernstein type theorem]{quasiconformal mappings and a Bernstein type theorem over exterior domains in $\mathbb{R}^2$}
\author{Dongsheng Li}
\author{Rulin Liu}
\begin{abstract}
  We establish the H\"{o}lder estimate and the asymptotic behavior at infinity
   for $K$-quasiconformal mappings over exterior domains in $\mathbb{R}^2$.
    As a consequence, we prove an exterior Bernstein type theorem for fully nonlinear uniformly elliptic equations of second order in $\mathbb{R}^2$.
\end{abstract}
\footnotetext{\textit{Key words and phrases.} Quasiconformal Mappings, Exterior Bernstein Type Theorem, Fully Nonlinear Elliptic Equations, Asymptotic Behavior.\newline}
\footnotetext{This research is supported by NSFC 12071365.\newline}
\footnotetext{Dongsheng Li\newline
              lidsh@mail.xjtu.edu.cn\newline
              School of Mathematics and Statistics, Xi'an Jiaotong University, Xi'an, P.R.China 710049.\newline}
\footnotetext{Rulin Liu\newline
              lrl001@stu.xjtu.edu.cn\newline
              School of Mathematics and Statistics, Xi'an Jiaotong University, Xi'an, P.R.China 710049.}
\maketitle
\section{Introduction}
\indent For Bernstein type theorems for fully nonlinear elliptic equations, a famous theorem of J\"{o}rgens \cite{JO} asserts that any solution of the Monge-Amp\`{e}re equation
\begin{equation}
 \label{eq:M-A}
  \det{D^2u}=1
\end{equation}
in $\mathbb{R}^2$ is a quadratic polynomial. This result was proved to be valid in higher dimensions by Calabi ($n\leq 5$ \cite{CA}) and Pogolove ($n\geq 2$ \cite{PO}) for convex $u$. The extension to the classical theorem of J\"{o}rgens, Calabi and Pogorelov above follows. In 2003, Caffarelli and Li\cite{C-LI} proved that for $n\geq 3$, any convex viscosity solution of the Monge-Amp\`{e}re equation (\ref{eq:M-A}) outside a bounded open subset of $\mathbb{R}^n$ approaches a quadratic polynomial near infinity and for $n=2$, any viscosity solution tends to a quadratic polynomial plus a logarithm term, where for the later case, Ferrer, Mart\'{i}nez and Mil\'{a}n \cite{F-F-M} obtained the same result in 1999. For the case of half space $\mathbb{R}^n_+$, Savin \cite{SA} established the Bernstein type theorem for equation (\ref{eq:M-A}) in 2014, and later, in 2020, Jia, Li and Li\cite{JIA-LI-LI} extended this theorem to exterior domains in half space.\newline
\indent By the virtue of the Evans-Krylov estimate, we can see that for $n\geq3$, any smooth entire solution of the general fully nonlinear elliptic equation
\begin{eqnarray}
  \label{eq:F(D^2u)}
  F\left(D^2u(x)\right)=0
\end{eqnarray}
in $\mathbb{R}^n$ is a quadratic polynomial if we assume the concavity of $F$ and the boundedness of the Hessian $D^2u$. For $n=2$, the same conclusion follows from the Nirenberg estimate \cite{NI} and the boundedness of $D^2u$ without the concavity of $F$.\newline
\indent In 2020, Li, Li and Yuan\cite{LI-LI-Y} established a higher dimensional exterior Bernstein type theorem for the fully nonlinear elliptic equation (\ref{eq:F(D^2u)}), namely, for $n\geq 3$, the solution of (\ref{eq:F(D^2u)}) in $\mathbb{R}^n\setminus\bar{B}_1(0)$ tends to a quadratic polynomial as $|x|\to\infty$ if $F$ is convex (or concave or the level set of $F$ is convex) and $D^2u$ is bounded. As applications of this theorem, the authors obtained the exterior Bernstein type theorems of Monge-Amp\`{e}re equations, special Lagrangian equations, quadratic Hessian equations and inverse harmonic Hessian equations for $n\geq 3$. As for $n=2$, the authors studied these three specific equations one by one to obtain the corresponding exterior Bernstein type theorem instead of establishing the general theorem to equation (\ref{eq:F(D^2u)}). Indeed, the method in \cite{LI-LI-Y} does not work for two dimensional problems. Roughly speaking, there are two steps in \cite{LI-LI-Y} to establish the exterior Bernstein type theorem.
First, by the concavity of $F$ and the boundedness of $D^2u$, the authors made use of the Evans-Krylov estimate and the weak Harnack inequality to show the existence of the limit $A$ of $D^2u$ at infinity, which actually holds for all $n\geq 2$. Second, it is crucial to get the decay rate of $|D^2u-A|$ as $|x|\rightarrow\infty$. This can be done by using barrier functions as $n\geq 3$ while unfortunately, such barrier does not exist as $n=2$.
\newline
\indent In this paper, we establish the exterior Bernstein type theorem for fully nonlinear elliptic equation (\ref{eq:F(D^2u)}) in $\mathbb{R}^2$ by using $K$-quasiconformal mappings. The main result goes as the following.
\begin{thm}
 \label{thm:1}
 Let $u$ be a viscosity solution of (\ref{eq:F(D^2u)}) in the exterior domain
 $\mathbb{R}^2\setminus \bar{\Omega}$, where $F\in C^{1,1}$ is a fully nonlinear uniformly elliptic operator with ellipticity constants $\lambda$ and $\Lambda$, and $\Omega$ is a bounded domain of $\mathbb{R}^2$. If
 $\left\| D^2u\right\|_{L^{\infty}(\mathbb{R}^2\setminus \bar{\Omega})}\leq M< +\infty$,
then there exists a unique symmetric matrix $A\in \mathbb{R}^{2\times 2}$, $b,e\in \mathbb{R}^2, c,d\in\mathbb{R}$
 such that for any $0<\alpha<1$,
 $$u(x)=\frac 12 x^{\mathrm{T}}Ax+b\cdot x+d\log{|x|}+c+e\frac{x}{|x|^2}+O\left(|x|^{-1-\alpha}\right)\text{~as~}|x|\to\infty,$$
 where
\begin{equation}
 \label{value of d}
 d=\frac {1}{2\pi}\left(\int\limits_{\partial\Omega} u_{\nu}\mathrm{d}s+\iint\limits_{\mathbb{R}^2\setminus\bar{\Omega}} (\Delta u(x)-\mathrm{tr}A)\mathrm{d}x_1\mathrm{d}x_2-\mathrm{tr}A|\Omega|\right),
\end{equation}
 $\nu$ is the unit outward normal of the boundary $\partial\Omega$.
 Furthermore, if $F$ is smooth, then we have
 $$\left|D^k\left(u(x)-\frac 12 x^{\mathrm{T}}Ax-b\cdot x-d\log{|x|}-c-e\frac{x}{|x|^2}\right)\right|=O\left(|x|^{-1-\alpha-k}\right)\text{~as~}|x|\to\infty$$
 for all $k\in\mathbb{N}$.
\end{thm}
\begin{re}
In Theorem \ref{thm:1}, the concavity (or convexity or convexity of the level set $\{N|F(N)=0\}$) of $F$ is not needed that is however an essential assumption in \cite{LI-LI-Y}.
\end{re}
As aforementioned, we will use $K$-quasiconformal mappings to study equation (\ref{eq:F(D^2u)}) over exterior domains. $K$-quasiconformal mappings play a special role in studying the H\"{o}lder continuity of solutions of two dimensional second order partial differential equations, which was developed by Morrey \cite{MO}, Nirenberg \cite{NI} and Finn and Serrin \cite{F-S}. In this paper, we will demonstrate the asymptotic behavior of $K$-quasiconformal mappings at infinity over exterior  domains (Cf. Theorem \ref{thm:Holder of exterior q-c} in Section 2). By using this result to (\ref{eq:F(D^2u)}) over exterior domains,
we shall not only show $D^2u$ has a limit $A$ at infinity, but  get the decay rate of $|D^2u-A|$ as $|x|\rightarrow\infty$ as well. After this, Theorem 1.1 will be proved by standard arguments.
\newline
\indent The organization of this paper goes as follows. In section 2, we study the H\"{o}lder continuity and asymptotic behavior at infinity of $K$-quasiconformal mappings over exterior domains, which implies the gradient H\"{o}lder estimate and the gradient asymptotic behavior at infinity of solutions of linear elliptic equations over exterior domains. In section 3, we give the proof of Theorem \ref{thm:1}.
\section{Exterior $K$-quasiconformal mappings}
Let's begin with the definition of exterior $K$-quasiconformal mappings in $\mathbb{R}^2\setminus\bar{\Omega}$. We refer to \cite{G-T} for the original definition of $K$-quasiconformal mappings.
\begin{defn}
 \label{def:quasiconformal exterior}
A mapping $w(x)=(p(x),q(x))$ from $\mathbb{R}^2\setminus\bar{\Omega}$ $(\Omega\subset\mathbb{R}^2\mbox{~is bounded~})$ in $x=(x_1,x_2)$ plane to $w=(p,q)$ plane is exterior $K$-quasiconformal in $\mathbb{R}^2\setminus\bar{\Omega}$ if $p, q\in C^1\left(\mathbb{R}^2\setminus\bar{\Omega}\right)$ and
\begin{eqnarray}
 \label{eq:q-c}
 p^2_1+p^2_2+q^2_1+q^2_2\leq 2K\left(p_1q_2-p_2q_1\right)
\end{eqnarray}
holds for all $x\in\mathbb{R}^2\setminus\bar{\Omega}$ with some constant $K>0$, where $p_i=\frac{\partial p(x)}{\partial x_i},q_i=\frac{\partial q(x)}{\partial x_i}, i=1,2$.
\end{defn}
For $K$-quasiconformal mappings, the apriori interior H\"{o}lder estimate is well known (Cf. \cite[Lemma 2]{NI} and \cite[Theorem 1]{F-S}).\newline
\indent For exterior $K$-quasiconformal mappings, we have the following H\"{o}lder estimate over exterior domain and the asymptotic behavior at infinity.
\begin{thm}
 \label{thm:Holder of exterior q-c}
Let $w=(p,q)$ be exterior $K$-quasiconformal in $\mathbb{R}^2\setminus \bar{\Omega}$ $(\Omega\subset\mathbb{R}^2$ is bounded $)$ with $K\geq 1$, and suppose $|w|\leq M$. Then, for any $\Omega^{\prime}\supset\supset\Omega$ with $d=\mathrm{dist}(\Omega,\partial\Omega^{\prime})$,
\begin{equation*}
 \left|w(x)-w(y)\right|\leq C\left|x-y\right|^{\alpha},x,y\in\mathbb{R}^2\setminus\overline{\Omega^{\prime}}.
\end{equation*}
and $w(x)$ tends to a limit $w(\infty)$ at infinity such that
\begin{equation}
 \label{asymptotic of exterior q-c}
 \left|w(x)-w(\infty)\right|\leq C|x|^{-\alpha}\mbox{~for any~}x\in\mathbb{R}^2\setminus\overline{\Omega^{\prime}},
\end{equation}
where $\alpha=K-(K^2-1)^{\frac 12}$, $C$ depends only on $K, d$ and $M$.
\end{thm}
\begin{re}
The results in Theorem \ref{thm:Holder of exterior q-c} are also valid for $p,q\in W_{\mathrm{loc}}^{1,2}(\mathbb{R}^2\setminus\bar{\Omega})\cap L^{\infty}(\mathbb{R}^2\setminus\bar{\Omega})$.
\end{re}
To prove Theorem \ref{thm:Holder of exterior q-c}, we first state the following H\"{o}lder continuity of $K$-quasiconformal mappings with singularities.
\begin{lemma}[\textnormal{\cite[Theorem 3]{F-S}}]
 \label{le:interior Holder of q-c}
Let $w=(p,q)$ be $K$-quasiconformal in a domain $\Omega$ of $x=
(x_1,x_2)$ plane, except at a set $T$ of isolated points in $\Omega$. Assume $|w|\leq M$. Then $w$ can be defined, or redefined, at the points of $T$ so that the resulting function is continuous in $\Omega$, and in any compact subregion $\Omega^{\prime}$ of $\Omega$ with $d=\mathrm{dist}(\Omega^{\prime}, \partial\Omega)$, $w(x)$ satisfies a uniform H\"{o}lder inequality
\begin{equation}
 \label{interior Holder of q-c}
 |w(x)-w(y)|\leq C|x-y|^{\alpha}, x,y\in\Omega^{\prime},
\end{equation}
where $\alpha=K-(K^2-1)^{\frac 12}$, $C$ depends only on $K, d$ and $M$.
\end{lemma}
We prove Theorem \ref{thm:Holder of exterior q-c} by making use of the Kelvin transform. For this purpose, we establish the following lemma, which states that the Kelvin transform of an exterior $K$-quasiconformal mapping is $K$-quasiconformal with an isolated singularity.
\begin{lemma}
 \label{le:transform}
 Let $w=(p,q)$ be exterior $K$-quasiconformal in $\mathbb{R}^2\setminus\bar{B}_1(0)$. Let $\tilde{p}$ and $\tilde{q}$ be the \textit{Kelvin transform} of $p$ and $q$ respectively, namely
 $$\tilde{p}(x)=p\left(\frac x{|x|^2}\right),\tilde{q}(x)=q\left(\frac x{|x|^2}\right), x\in B_1(0)\setminus\{0\}.$$
 Then, $\tilde{w}=(\tilde{q},\tilde{p})$ is $K$-quasiconformal in $B_1(0)\setminus\{0\}$.
\end{lemma}
\begin{proof}
Calculating directly, we have
$$\tilde{p}_1=\left(|x|^{-2}-2x^2_1|x|^{-4}\right)p_1+\left(-2x_1x_2|x|^{-4}\right)p_2,$$
$$\tilde{p}_2=\left(-2x_1x_2|x|^{-4}\right)p_1+\left(|x|^{-2}-2x^2_2|x|^{-4}\right)p_2,$$
$$\tilde{q}_1=\left(|x|^{-2}-2x^2_1|x|^{-4}\right)q_1+\left(-2x_1x_2|x|^{-4}\right)q_2,$$
and
$$\tilde{q}_2=\left(-2x_1x_2|x|^{-4}\right)q_1+\left(|x|^{-2}-2x^2_2|x|^{-4}\right)q_2.$$
It's easy to see that
$$\tilde{p}^2_1+\tilde{p}^2_2+\tilde{q}^2_1+\tilde{q}^2_2=|x|^{-4}\left(p^2_1+p^2_2+q^2_1+q^2_2\right),$$
and
$$\tilde{p}_1\tilde{q}_2-\tilde{p}_2\tilde{q}_1=-|x|^{-4}\left(p_1q_2-p_2q_1\right).$$
Since $w=(p,q)$ is exterior $K$-quasiconformal over $\mathbb{R}^2\setminus\bar{B}_1(0)$, we deduce by Definition \ref{def:quasiconformal exterior} that $p$ and $q$ satisfy (\ref{eq:q-c}) in $\mathbb{R}^2\setminus\bar{B}_1(0)$ for some $K\geq 1$.
So, we obtain that in $B_1(0)\setminus\{0\}$,
$$\tilde{p}^2_1+\tilde{p}^2_2+\tilde{q}^2_1+\tilde{q}^2_2\leq 2K\left(\tilde{p}_2\tilde{q}_1-\tilde{p}_1\tilde{q}_2\right),$$
which implies $\tilde{w}=(\tilde{q},\tilde{p})$ is $K$-quasiconformal in $B_1(0)\setminus\{0\}$.
\end{proof}
\begin{proof}[Proof of Theorem \ref{thm:Holder of exterior q-c}] Assume without loss of generality that $B_1(0)\subset\Omega$. Let $\tilde{p}$ and $\tilde{q}$ be the Kelvin transform of $p$ and $q$ respectively given by Lemma \ref{le:transform}. Let $\hat{\Omega}=\left\{\frac {x}{|x|^2}\Big|x\in\mathbb{R}^2\setminus\bar{\Omega}\right\}$ and for any $\Omega^{\prime}\supset\supset\Omega$, $\tilde{\Omega}=\left\{\frac {x}{|x|^2}\Big|x\in\mathbb{R}^2\setminus\overline{\Omega^{\prime}}\right\}$. Then by Lemma \ref{le:transform}, $\tilde{w}=(\tilde{q},\tilde{p})$ is $K$-quasiconformal in $\hat{\Omega}\setminus\{0\}$ with $K\geq 1$. Since $|w|\leq M$ implies $|\tilde{w}|\leq M$, applying Lemma \ref{le:interior Holder of q-c} to $\tilde{w}$ with $T=\{0\}$, we know that
$$|\tilde{w}(x)-\tilde{w}(y)|\leq C|x-y|^{\alpha}, x,y\in{\tilde{\Omega}},$$
which implies that $\tilde{w}(x)$ has a limit $\tilde{w}(0)$ at $0$ and for all $x\in{\tilde{\Omega}},$
$$\left|\tilde{w}(x)-\tilde{w}(0)\right|\leq C|x|^{\alpha}, \alpha=K-\left(K^2-1\right)^{\frac 12}.$$
Transforming back to exterior domain, we have that
$$|w(x)-w(y)|\leq C|x-y|^{\alpha}, x,y\in\mathbb{R}^2\setminus\overline{\Omega^{\prime}}$$
and $w(x)$ has a limit $w(\infty)=\tilde{w}(0)$ at infinity with
$$\left|w(x)-w(\infty)\right|\leq C|x|^{-\alpha},x\in\mathbb{R}^2\setminus\overline{\Omega^{\prime}},$$
where $\alpha=K-(K^2-1)^{\frac 12}$, $C$ depends only on $K, d$ and $M$.\newline
\indent The theorem is therefore proved.
\end{proof}
Next we consider linear elliptic equation
\begin{equation}
 \label{eq:L}
 L(u)=a_{11}(x)u_{11}(x)+2a_{12}(x)u_{12}(x)+a_{22}(x)u_{22}(x)=0,
\end{equation}
where $L$ is uniformly elliptic, that is, there exist $0<\lambda\leq\Lambda$ such that
\begin{equation}
 \label{uniform elliptic 1}
 \lambda(\xi_1^2+\xi_2^2)\leq a_{11}\xi_1^2+2a_{12}\xi_1\xi_2+a_{22}\xi_2^2\leq \Lambda(\xi_1^2+\xi_2^2), \forall \xi=(\xi_1,\xi_2)\in\mathbb{R}^2
\end{equation}
and
\begin{equation}
 \label{uniform elliptic 2}
 \frac{\Lambda}{\lambda}\leq \gamma
\end{equation}
for some constant $\gamma\geq 1$.\newline
\indent For uniformly elliptic equation (\ref{eq:L}) in a domain $\Omega$ of $\mathbb{R}^2$, it follows from the interior H\"{o}lder estimate of $K$-quasiconformal mappings that its bounded solutions have interior $C^{1,\alpha}$ estimate\cite[Theorem 12.4]{G-T}.\newline
\indent For uniformly elliptic equation (\ref{eq:L}) over exterior domain in $\mathbb{R}^2$, we can establish the gradient H\"{o}lder estimate and the gradient asymptotic behavior of solutions at infinity by the virtue of Theorem \ref{thm:Holder of exterior q-c}.
\begin{thm}
 \label{thm:exterior for equation}
  Let $\Omega$ be a bounded domain of $\mathbb{R}^2$ and $u\in C^2(\mathbb{R}^2\setminus \bar{\Omega})$ be a solution of equation (\ref{eq:L}) in $\mathbb{R}^2\setminus \bar{\Omega}$. Suppose $|Du(x)|\leq M$. Then for any $\Omega^{\prime}\supset\supset\Omega$ with $d=\mathrm{dist}(\Omega,\partial\Omega^{\prime})$,
  \begin{equation*}
   \left|Du(x)-Du(y)\right|\leq C\left|x-y\right|^{\alpha}, x,y\in\mathbb{R}^2\setminus\overline{\Omega^{\prime}}
  \end{equation*}
  and $Du(x)$ has a limit $Du(\infty)$ at infinity with
  \begin{equation}
   \label{asymptotic of Du}
   |Du(x)-Du(\infty)|\leq C|x|^{-\alpha},  x\in\mathbb{R}^2\setminus\overline{\Omega^{\prime}},
  \end{equation}
where $\alpha$ depends only on $\gamma$, $C$ depends only on $\gamma, d$ and $M$.
\end{thm}
\begin{re}
The results in Theorem \ref{thm:exterior for equation} are also valid for $u\in W^{2,2}(\mathbb{R}^2\setminus\bar{\Omega})$.
\end{re}
\begin{proof}[Proof of Theorem \ref{thm:exterior for equation}]
\indent Assume without loss of generality that $\lambda=1$. Let $p=u_1, q=u_2$. By equation (\ref{eq:L}), (\ref{uniform elliptic 1}) and (\ref{uniform elliptic 2}), we have (see details in \cite{G-T})
$$p^2_1+p^2_2\leq a_{11}p^2_1+2a_{12}p_1p_2+a_{22}p^2_2=a_{22}J, J=p_2q_1-p_1q_2, x\in \mathbb{R}^2\setminus \bar{\Omega}$$
and
$$q^2_1+q^2_2\leq a_{11}J, x\in \mathbb{R}^2\setminus \bar{\Omega}.$$
Noticing that $2\leq a_{11}+a_{22}=1+\Lambda\leq 1+\gamma$, we arrive at
$$p^2_1+p^2_2+q^2_1+q^2_2\leq \left(a_{11}+a_{22}\right)J\leq \left(1+\gamma\right)J, x\in \mathbb{R}^2\setminus \bar{\Omega},$$
which implies that $w=(q,p)$ is exterior $K$-quasiconformal over $\mathbb{R}^2\setminus \bar{\Omega}$ with $K=\frac{1+\gamma}{2}$. Since $|Du|\leq M$ in $\mathbb{R}^2\setminus\bar{\Omega}$, Theorem \ref{thm:Holder of exterior q-c} therefore asserts that for any $\Omega^{\prime}\supset\supset\Omega$,
$$|Du(x)-Du(y)|\leq C|x|^{-\alpha},  x,y\in\mathbb{R}^2\setminus\overline{\Omega^{\prime}}$$
and $Du(x)$ tends to a limit $Du(\infty)=(p(\infty), q(\infty))$ at infinity with
$$|Du(x)-Du(\infty)|\leq C|x|^{-\alpha}, x\in\mathbb{R}^2\setminus\overline{\Omega^{\prime}},$$
where $\alpha$ depends only on $\gamma$, $C$ depends only on $\gamma, d$ and $M$.
\end{proof}
\section{Exterior Bernstein type theorem}
In this section, we give the proof of the exterior Bernstein type theorem, i.e., Theorem \ref{thm:1}. As we remarked before, we don't need the concavity or convexity of $F$.\newline
\indent We find the limit $A$ of the Hessian $D^2u$ at infinity and estimate the decay rate of $|D^2u-A|$ first.
\begin{thm}
 \label{thm:limit of D^2u}
Let $u$ be as in Theorem \ref{thm:1}. Then there exists a symmetric matrix
$A\in\mathbb{R}^{2\times2}$ such that
$$D^2u(x)\to A\text{~as~}|x|\to\infty$$ and
$$|D^2u(x)-A|\leq C|x|^{-\alpha}\text{~as~}|x|\to\infty,$$
which implies
\begin{eqnarray}
 \label{in:u-xAx}
 \left|u(x)-\frac 12 x^{\mathrm{T}}Ax\right|\leq C|x|^{2-\alpha}\text{~as~}|x|\to\infty,
\end{eqnarray}
where $\alpha\in(0,1)$ is a constant depending only on $\lambda$ and $\Lambda$, $C$ is a positive constant depending only on $\lambda$, $\Lambda$, and $M$.
\end{thm}
\begin{re}
If $u\in C^2$, then we don't need $F\in C^{1,1}$ in Theorem \ref{thm:limit of D^2u}.
\end{re}
\begin{proof}[\textit{Proof of Theorem \ref{thm:limit of D^2u}}]
By the virtue of the Nirenberg estimate, we can see that viscosity solutions to the equation (\ref{eq:F(D^2u)}) in $\mathbb{R}^2$ are always $C^{2,\alpha}$ for some $\alpha\in(0,1)$ depending only on the ellipticity constants of $F$.  It follows from $F\in C^{1,1}$ and the Schauder estimate that $u\in C^{3,\gamma}(\mathbb{R}^2\setminus\bar{\Omega})$ for any $\gamma\in(0,1)$.
Then we take derivative with respect to $x_k$ $(k=1,2)$ on both sides of equation (\ref{eq:F(D^2u)}) to obtain
\begin{align}
 \label{eq:aij}
 a_{ij}(x)v_{ij}(x)=0, x\in\mathbb{R}^2\setminus\bar{\Omega},
\end{align}
where $a_{ij}(x)=F_{M_{ij}}\left(D^2u(x)\right)$ and $v(x)=u_k(x)$.
\newline
\indent Since $\|D^2u\|_{L^{\infty}\left(\mathbb{R}^2\setminus\bar{\Omega}\right)}\leq M$, we know $|Dv(x)|\leq M$. Applying Theorem \ref{thm:exterior for equation} to equation (\ref{eq:aij}) in $\mathbb{R}^2\setminus\bar{\Omega}$, we have that $Dv(x)$ tends to a limit $Dv(\infty)$ at infinity and for any $\Omega^{\prime}\supset\supset\Omega$,
$$|Dv(x)-Dv(\infty)|\leq C|x|^{-\alpha}, x\in\mathbb{R}^2\setminus\overline{\Omega^{\prime}}.$$
Then by the arbitrarity of $k$, we conclude that there exists a symmetric matrix $A\in\mathbb{R}^{2\times 2}$ such that $D^2u(x)\to A$ as $|x|\to\infty$ and
$$\left|D^2u(x)-A\right|\leq C|x|^{-\alpha} \text{~as~}|x|\to\infty.$$
It follows that
$$\left|u(x)-\frac 12 x^{\mathrm{T}}Ax\right|\leq C|x|^{2-\alpha}\text{~as~}|x|\to\infty,$$
where $\alpha\in(0,1)$ depends only on $\lambda$ and $\Lambda$, $C>0$ depends only on $\lambda$, $\Lambda$ and $M$.
\end{proof}
\indent Based on Theorem \ref{thm:limit of D^2u}, we will find the finer asymptotic behavior of $u$ by standard arguments. To do this, we need the following three lemmas which are well known. For readers' convenience, we show the proofs of them. Lemma \ref{le:smoothness} gives the higher order estimates. Lemma \ref{le:iterate} and Lemma \ref{le:expansion} are used to determine the linear term, logarithm term and constant term of the asymptotics of $u$.
\begin{lemma}
 \label{le:smoothness}
 Let $\phi$ be a viscosity solution of the equation
 $$F\left(D^2\phi(x)+A\right)=0, x\in\mathbb{R}^2\setminus \bar{B}_1(0),$$
 where $F\in C^{1,1}$ is a fully nonlinear uniformly elliptic operator with ellipticity constants $\lambda$ and $\Lambda$, and $A\in\mathbb{R}^{2\times 2}$ is symmetric matrix, satisfying $F(A)=0$. Suppose that for some constants $\beta>0$ and $\rho<2$,
 $$|\phi(x)|\leq \beta|x|^{\rho}, x\in\mathbb{R}^2\setminus \bar{B}_1(0).$$
 Then there exists some constant $r=r(\beta, \rho)\geq 1$ such that for $k=0,1,2,3$,
 $$\left|D^k\phi(x)\right|\leq C|x|^{\rho-k}, x\in\mathbb{R}^2\setminus \bar{B}_r(0),$$
 where $C$ depends only on $\lambda$, $\Lambda$, $\beta$ and $\rho$.
\end{lemma}
\begin{proof}
By $F\in C^{1,1}$, the Nirenberg estimate and the Schauder estimate, $\phi(x)\in C^{3,\gamma}$ for any $\gamma\in(0,1)$.

Fix $x\in\mathbb{R}^2\setminus\bar{B}_1(0)$ with $|x|>6$ and let
$$\bar{\phi}(y)=\left(\frac {2}{|x|}\right)^2 \phi\left(x+\frac{|x|}{2}y\right), y\in B_1(0).$$
\indent Since $$F(A)=0$$ and $$F(D^2\bar{\phi}(y)+A)=0, y\in B_1(0),$$ we see that
$$\bar{a}_{ij}(y)\bar{\phi}_{ij}(y)=0, y\in B_1(0),$$
where $\bar{a}_{ij}(y)=\int_0^1F_{M_{ij}}\left(tD^2\bar{\phi}(y)+A\right)\mathrm{d}t$. By the Schauder estimate, we have that for $k=0,1,2,3$,
$$\left|D^k\bar{\phi}(0)\right|\leq\|\bar{\phi}\|_{L^{\infty}(\bar{B}_1(0))}\leq C|x|^{\rho-2},$$
which implies
$$\left|D^k\phi(x)\right|\leq C|x|^{\rho-k},$$
where $C$ depends only on $\lambda$, $\Lambda$, $\beta$ and $\rho$.
\end{proof}
\begin{lemma}
 \label{le:iterate}
Suppose $f(x)=O(|x|^{-\beta})$ as $|x|\to\infty$ with $\beta>1$. Then for any $\varepsilon>0$, the equation
$$\Delta u(x)=f(x)\text{~in~}\mathbb{R}^2\setminus \bar{B}_1(0)$$
has a solution $u(x)=O(|x|^{2-\beta+\varepsilon})$ as $|x|\to\infty$.
\end{lemma}
\begin{proof}
Let $$u(x)=-\frac {1}{2\pi}\int\limits_{\mathbb{R}^2\setminus \bar{B}_1(0)}(\log{|x-y|-\log{|y|}})f(y)\mathrm{d}y.$$
Then
$$\Delta u(x)=f(x), x\in\mathbb{R}^2\setminus \bar{B}_1(0)$$
and for any $\varepsilon>0$,
$$\left|u(x)\right|\leq C(\varepsilon)|x|^{2-\beta+\varepsilon}, x\in\mathbb{R}^2\setminus \bar{B}_1(0).$$
\end{proof}
\begin{lemma}
 \label{le:expansion}
 Let $u(x)=O(|x|^{\beta})$ be a smooth solution of $$\Delta u(x)=0, x\in\mathbb{R}^2\setminus \bar{B}_1(0)$$ for some $0<\beta<2$.
 Then
 \begin{equation}
  \label{expansion of u}
  u=b\cdot x+d\log{|x|}+c+O\left(|x|^{-1}\right)\mbox{~as~}|x|\to\infty,
 \end{equation}
 where $b\in\mathbb{R}^2$, $c, d\in\mathbb{R}$. Particularly, for $0<\beta<1$, (\ref{expansion of u}) holds with $b=0$.
\end{lemma}
\begin{proof}
Let $\xi(z)=u_1(x)-iu_2(x), z=x_1+ix_2$. Then $\xi(z)$ is an analytic function in $\mathbb{R}^2\setminus \bar{B}_1(0)$ and the growth of $\xi(z)$ is at most of order $|z|^{\beta-1}$. Since $0<\beta<2$, the Laurent expansion of $\xi(z)$ has the form
\begin{equation}
 \label{laurent expasion}
 \xi(z)=a_0+a_{-1}z^{-1}+a_{-2}z^{-2}+\cdots, z\in\mathbb{R}^2\setminus\bar{B}_1(0),
\end{equation}
where $a_0, a_{-1}, a_{-2}, \cdots$ are all complex numbers. Thus we have
$$Du(x)=D(b\cdot x+c_1)+D(a_{-1}\log{|x|}+c_2)+O(|x|^{-2})\text{~as~}|x|\to\infty,$$
where $b=(\mathrm{Re}~a_0, -\mathrm{Im}~a_0)^{\mathrm{T}}, c_1, c_2\in\mathbb{R}$. Since $\mathrm{Re}\int a_{-1}z^{-1}=\mathrm{Re}(a_{-1}\log{z})$ as a part of expansion of a real function $u$, $a_{-1}$ must be a real number.
Integrating the above, we see that
$$u=b\cdot x+d\log{|x|}+c+O(|x|^{-1})\text{~as~}|x|\to\infty,$$
where $c\in\mathbb{R}, d=a_{-1}\in\mathbb{R}$.\newline
\indent Particularly, for $0<\beta<1$, (\ref{laurent expasion}) holds with $a_0=0$. Therefore, the above equality also holds with $b=0$.
\end{proof}
\begin{proof}[Proof of Theorem \ref{thm:1}]We divide the proof into six steps.
\ \newline
\indent\textit{Step 1. Improving estimate (\ref{in:u-xAx}).}\newline
\indent Let $$\varphi(x)=u(x)-\frac 12 x^{\mathrm{T}}Ax.$$
Then by Theorem \ref{thm:limit of D^2u}, $$\varphi(x)=O(|x|^{2-\alpha})$$ and $\varphi(x)$ satisfies
\begin{equation}
 \label{eq:F(D^2 varphi+A)}
 F(D^2\varphi(x)+A)=0, x\in\mathbb{R}^2\setminus\bar{\Omega}.
\end{equation}
Suppose $R_0\geq 1$ such that $\Omega\subset B_{R_0}(0)$. It follows from Lemma \ref{le:smoothness} that for all $x\in\mathbb{R}^2\setminus\bar{B}_{R_0}(0)$,
\begin{eqnarray}
 \left|D\varphi(x)\right|\leq C|x|^{1-\alpha},
 \left|D^2\varphi(x)\right|\leq C|x|^{-\alpha},
 \left|D^3\varphi(x)\right|\leq C|x|^{-1-\alpha}.
 \label{deri}
\end{eqnarray}
\indent Taking derivative to both sides of equation (\ref{eq:F(D^2 varphi+A)}) with respect to $x_k$ $(k=1,2)$, we know that $\varphi_k$ satisfies equation
\begin{eqnarray}
 a_{ij}(x)\left(\varphi_k(x)\right)_{ij}=0, x\in \mathbb{R}^2\setminus\bar{B}_{R_0}(0),
 \label{eq:linear 1}
\end{eqnarray}
where $a_{ij}(x)=F_{M_{ij}}\left(D^2\varphi(x)+A\right)$. Since it follows from Theorem \ref{thm:limit of D^2u} that $D^2\varphi(x)\to 0$ as $|x|\to\infty$, we know
\begin{equation}
a_{ij}(x)\to F_{M_{ij}}(A)\mbox{~as~}|x|\to\infty.
\end{equation}
Assuming without loss of generality that $F_{M_{ij}}(A)=\delta_{ij}$, then by $F\in C^{1,1}$,
\begin{equation}
\left|\delta_{ij}-a_{ij}\right|\leq C|x|^{-\alpha}
\end{equation}
for some $C>0$. We obtain that for all $x\in\mathbb{R}^2\setminus\bar{B}_{R_0}(0)$,
$$\varphi_k(x)=O(|x|^{1-\alpha})$$ and
\begin{equation}
\label{eq:Delta varphi k}
\Delta(\varphi_k)(x)=\left(\delta_{ij}-a_{ij}(x)\right)(\varphi_k)_{ij}(x)=O(\left|x|^{-\alpha}|x|^{-1-\alpha}\right)=O\left(|x|^{-1-2\alpha}\right). \end{equation}
\indent By Lemma \ref{le:iterate}, for any $0<\varepsilon<\alpha$, there exists $$v(x)=O(|x|^{1-2\alpha+\varepsilon})$$
satisfying the equation (\ref{eq:Delta varphi k}). Then
\begin{equation}
\Delta (\varphi_k-v)(x)=0, x\in\mathbb{R}^2\setminus\bar{B}_{R_0}(0)
\end{equation}
and
$$\varphi_k(x)-v(x)=O(|x|^{1-\alpha}).$$
Therefore Lemma \ref{le:expansion} states
$$\varphi_k(x)-v(x)=d\log{|x|}+c+O\left(|x|^{-1}\right)\mbox{~as~}|x|\to\infty$$
for some $b\in\mathbb{R}^2, c\in\mathbb{R}$. Hence, for $k=1,2$, $$\varphi_k(x)=O(|x|^{1-2\alpha+\varepsilon}).$$
By the arbitrarity of $k$, we see $$\varphi(x)=O(|x|^{2-2\alpha+\varepsilon}).$$
Since $0<\varepsilon<\alpha$, we have improved the estimate (\ref{in:u-xAx}) a little.\newline
\indent We repeat the arguments above $n$ times, where $n$ is determined by the following way. Fix $0<\varepsilon<\alpha$ and let $n$ be an integer such that $0<1-2^n\alpha+(2^n-1)\varepsilon<\frac 18$, i.e. $n=\left[\log_2{\frac{\frac 78-\varepsilon}{\alpha-\varepsilon}}\right]+1$. Then we get an appropriate improved estimate
$$\varphi(x)=O(|x|^{2-2^n\alpha+(2^n-1)\varepsilon})=O(|x|^{1+\delta}), x\in\mathbb{R}^2\setminus\bar{B}_{R_0}(0)$$
with $\delta=1-2^n\alpha+(2^n-1)\varepsilon<\frac 18$.\newline
\indent\textit{Step 2. Determining the linear term.}\newline
\indent We obtain by Lemma \ref{le:smoothness} that for $\delta\in(0,\frac 18)$ and all $x\in\mathbb{R}^2\setminus\bar{B}_{R_0}(0)$,
$$|D\varphi(x)|\leq C|x|^{\delta}, |D^2\varphi(x)|\leq C|x|^{-1+\delta}.$$
Since $\varphi(x)$ satisfies equation
\begin{equation}
 \label{eq:linear 2}
 \bar{a}_{ij}(x)\varphi_{ij}(x)=0, x\in\mathbb{R}^2\setminus\bar{B}_{R_0}(0),
\end{equation}
where $\bar{a}_{ij}(x)=\int_0^1F_{M_{ij}}\left(tD^2\varphi(x)+A\right)\mathrm{d}t$, it follows from $F\in C^{1,1}$ that for some $C>0$,
$$|\bar{a}_{ij}(x)-\delta_{ij}|\leq C|x|^{-1+\delta}.$$
Thus
$$\Delta \varphi(x)=\left(\delta_{ij}-\bar{a}_{ij}(x)\right)\varphi_{ij}(x)=O \left(|x|^{-2+2\delta}\right), x\in\mathbb{R}^2\setminus\bar{B}_{R_0}(0).$$
Then Lemma \ref{le:iterate} implies that for any $\varepsilon\in(0,\frac 18)$, there exists $$v(x)=O(|x|^{2\delta+\varepsilon}),$$ satisfying
$$\Delta(\varphi-v)(x)=0, x\in\mathbb{R}^2\setminus\bar{B}_{R_0}(0).$$
Since $$\varphi(x)-v(x)=O(|x|^{1+\delta}),$$
it follows from Lemma \ref{le:expansion} that there exists $b\in\mathbb{R}^2$ such that
$$\varphi(x)-v(x)=b\cdot x+O(\log{|x|}).$$
Hence $$\varphi(x)=b\cdot x+O(|x|^{2\delta+\varepsilon}).$$
\indent\textit{Step 3. Determining the logarithm term and constant term.}\newline
\indent Let $$\bar{\varphi}(x)=u-\left(\frac 12x^{\mathrm{T}}Ax+b\cdot x\right).$$
Then $$\bar{\varphi}(x)=O(|x|^{2\delta+\varepsilon})$$ and $\bar{\varphi}(x)$ satisfies equation (\ref{eq:linear 2}).\newline
\indent By Lemma \ref{le:smoothness}, we see that for all $x\in\mathbb{R}^2\setminus\bar{B}_{R_0}(0)$,
$$|D\bar{\varphi}(x)|\leq C|x|^{-1+2\delta+\varepsilon}, |D^2\bar{\varphi}(x)|\leq C|x|^{-2+2\delta+\varepsilon}.$$
Consequently, for some $C>0$,
$$|\bar{a}_{ij}-\delta_{ij}|\leq C|x|^{-2+2\delta+\varepsilon}$$
and $$\Delta\bar{\varphi}(x)=(\delta_{ij}-\bar{a}_{ij})\bar{\varphi}_{ij}=O(|x|^{-4+4\delta+2\varepsilon}).$$
Since $\delta, \varepsilon\in (0,\frac 18)$, then by Lemma \ref{le:iterate}, there exists
$$v(x)=O(|x|^{-2+\varepsilon^{\prime}})$$
with $\varepsilon^{\prime}\in(0,1)$, satisfying
$$\Delta(\bar{\varphi}-v)(x)=0$$ and
$$\bar{\varphi}(x)-v(x)=O(|x|^{2\delta+\varepsilon}).$$
Thus, Lemma \ref{le:expansion} leads to
$$\bar{\varphi}(x)=d\log{|x|}+c+O\left(|x|^{-1}\right) \mbox{~as~}|x|\to\infty$$
for some $c, d\in\mathbb{R}$, namely,
\begin{equation}
 \label{asymptotic of u}
  u(x)=\frac 12 x^{\mathrm{T}}Ax+b\cdot x+d\log{|x|}+c+O(|x|^{-1}).
\end{equation}
\indent\textit{Step 4. Determining the $\frac {x}{|x|^2}$ term.}\newline
\indent Let $$\hat{\varphi}(x)=u(x)-\left(\frac 12 x^{\mathrm{T}}Ax+b\cdot x+d\log{|x|}+c\right).$$
Then $$D^2\hat{\varphi}=D^2u-A+O(|x|^{-2}).$$
By (\ref{asymptotic of u}), $D^2u=A+O(|x|^{-2})$, which implies $$\left|D^2\hat{\varphi}\right|=O(|x|^{-2}).$$ Since $\hat{\varphi}(x)$ satisfies equation (\ref{eq:linear 2}) with $\bar{a}_{ij}(x)=\int_0^1F_{M_{ij}}\left(t\left(D^2\hat{\varphi}(x)+D^2(d\log{|x|})\right)+A\right)\mathrm{d}t$, we have that for some $R_0\geq 1$ such that $\Omega\subset B_{R_0}(0)$, $$\Delta\hat{\varphi}(x)=(\bar{a}_{ij}(x)-\delta_{ij})\hat{\varphi}_{ij}(x)=:f(x)=O(|x|^{-2}|x|^{-2})=O(|x|^{-4}), x\in\mathbb{R}^2\setminus\bar{B}_{R_0}(0).$$
\indent Let $\psi(x)=\hat{\varphi}(\frac {x}{|x|^2})$ and $\tilde{f}(x)=f(\frac {x}{|x|^2})$ be the Kelvin transform of $\hat{\varphi}(x)$ and $f(x)$ respectively. Then we see $$\psi(x)=O(|x|)$$ and
$$\Delta\psi(x)=|x|^{-4}\tilde{f}(x)=:g(x)=O(1), x\in B_{\frac{1}{R_0}}(0).$$
\indent From $g\in L^p(B_{1/R_0}(0))$ for any $p>2$, it follows that $\psi(x)\in W^{2,p}(B_{1/R_0}(0))$ and hence $\psi(x)\in C^{1,\alpha}(B_{1/R_0}(0))$ for $\alpha=1-\frac{2}{p}\in(0,1)$. Then there exists $e\in\mathbb{R}^2$ and $\tilde{c}\in\mathbb{R}$ such that for some $C>0$,
$$\left|\psi(x)-(e\cdot x+\tilde{c})\right|\leq C|x|^{1+\alpha},x\in B_{\frac {1}{R_0}}(0).$$
Since $\psi(0)=0$ implies $\tilde{c}=0$, we go back to exterior domain to get
$$\left|\hat{\varphi}(x)-e\cdot\frac{x}{|x|^2}\right|\leq C|x|^{-1-\alpha},x\in\mathbb{R}^2\setminus\bar{B}_{R_0}(0),$$
which leads to
$$u=\frac 12 x^{\mathrm{T}}Ax+b\cdot x+d\log{|x|}+c+e\frac {x}{|x|^2}+O(|x|^{-1-\alpha}).$$
\indent\textit{Step 5. Calculating the value of $d$.}\newline
\indent Let $Q(x)=\frac 12 x^{\mathrm{T}}Ax+b\cdot x+c$. Then $$u(x)=Q(x)+d\log{|x|}+O(|x|^{-1})$$
and
$$\Delta(u-Q)(x)=O(|x|^{-3})$$ is integrable. Let $\nu$ be the unit outward normal of boundaries $\partial\Omega$ and $C_R=\partial B_R(0)$. Then by the divergence theorem, we have that for some $R>0$ large enough,
\begin{equation*}
 \begin{split}
 \iint\limits_{{B_R(0)\setminus\bar{\Omega}}}\Delta(u-Q)(x)\mathrm{d}x_1\mathrm{d}x_2&=\int\limits_{\partial({B_R(0)\setminus\bar{\Omega}})}(u-Q)_{\nu}\mathrm{d}s\\
 &=\int\limits_{C_R}(d\log{|x|}+O(|x|^{-1}))_{\nu}(x)\mathrm{d}s-\int\limits_{\partial\Omega}(u-Q)_{\nu}\mathrm{d}s\\
 &=d\int\limits_{C_R}\frac{x}{|x|^2}\cdot\nu\mathrm{d}s+O\left(\frac 1R\right)-\int\limits_{\partial\Omega}u_{\nu}\mathrm{d}s+\int_{\partial\Omega}Q_{\nu}\mathrm{d}s\\
 &=2\pi d+O\left(\frac 1R\right)-\int\limits_{\partial\Omega}u_{\nu}\mathrm{d}s+\int_{\Omega}\Delta Q\mathrm{d}x\\
 &=2\pi d+O\left(\frac 1R\right)-\int\limits_{\partial\Omega}u_{\nu}\mathrm{d}s+\mathrm{tr}A|\Omega|.
 \end{split}
\end{equation*}
Letting $R\to\infty$, we get (\ref{value of d}).\newline
\indent\textit{Step 6. Improving smoothness of the error.}\newline
\indent Furthermore, suppose $F$ is smooth. Let
$$\tilde{\varphi}(x)=u-\left(\frac 12 x^{\mathrm{T}}Ax+b\cdot x+d\log{|x|}+c+e\frac {x}{|x|^2}\right).$$
Then, the Schauder estimate asserts that for all $k\in\mathbb{N}$,
$$\left|D^k\tilde{\varphi}(x)\right|\leq C(k)|x|^{-1-\alpha-k}.$$
\indent We complete the proof of Theorem \ref{thm:1}.
\end{proof}
\begin{re}(\romannumeral 1). If the equation has some divergence structure, then we can obtain another representation for the constant $d$, for example, the Monge-Amp\`{e}re equations, the special Lagrangian equations and the inverse harmonic Hessian equations. We refer to \cite{C-LI} and \cite{LI-LI-Y} to see details.\newline
\indent(\romannumeral 2). By the virtue of Theorem \ref{thm:1}, we have expansion for the solutions to the Monge-Amp\`{e}re equations, the special Lagrangian equations and the inverse harmonic Hessian equations at infinity in $\mathbb{R}^2\setminus\bar{\Omega}$, namely, any solution tends to a quadratic polynomial plus a logarithm term and $e\frac {x}{|x|^2}$ with the error at least $|x|^{-1-\alpha}$, which is finer than the results in \cite{C-LI} and \cite{LI-LI-Y}.
\end{re}
\bibliographystyle{elsarticle-num}

\end{document}